\documentclass{article}
\usepackage{amsfonts}
\usepackage{amsmath}

\newtheorem{theorem}{Theorem}

\newtheorem{definition}[theorem]{Definition}

\newtheorem{lemma}[theorem]{Lemma}

\newenvironment{proof}[1][Proof]{\noindent\textbf{#1.} }{\ \rule{0.5em}{0.5em}}
\input{tcilatex}
\begin{document}

\title{\textbf{On the properties of circulant matrices involving Generalized
Tribonacci and Generalized Tribonacci-Lucas numbers}}
\author{Nazmiye Yilmaz\thanks{%
e mail: \ \textit{nzyilmaz@selcuk.edu.tr, yyazlik@nevsehir.edu.tr,
ntaskara@selcuk.edu.tr}} \\
Department of Mathematics, Science Faculty, \\
Selcuk University, Campus, Konya, Turkey \and Yasin Yazlik \\
Department of Mathematics, Faculty of Science and Letters, \\
Nevsehir Haci Bektas Veli University, Nevsehir, Turkey \and Necati Taskara \\
Department of Mathematics, Science Faculty, \\
Selcuk University, Campus, Konya, Turkey }
\maketitle

\begin{abstract}
In this paper, firstly, we define the Generalized Tribonacci-Lucas numbers.
In addition, by also defining circulant matrices $C_{n}(G)$ and $C_{n}(S)$
whose entries are Generalized Tribonacci and Generalized Tribonacci-Lucas
numbers, we compute spectral norms and determinants of these matrices.


\end{abstract}

\section{Introduction}

Shannon and Horadam \cite{5} defined Generalized Tribonacci sequence $%
\left\{ G_{n}\right\} _{n\in 
\mathbb{N}
}$\ as 
\begin{equation}
G_{n+3}=PG_{n+2}+QG_{n+1}+RG_{n},\ \ \ \ G_{0}=0,~G_{1}=0,\ G_{2}=1,
\label{1}
\end{equation}%
where $P,\ Q,$ and $R$ are arbitrary integers.

If the characteristic equation%
\begin{equation*}
x^{3}-Px^{2}-Qx-R=0
\end{equation*}%
has three distinct real roots, suppose that they are given by $\alpha ,\
\beta ,\ \gamma $. According to the general theory of recurrence relations, $%
G_{n}$ can be represented by

\begin{equation}
G_{n}=\frac{\alpha ^{n}}{\left( \alpha -\beta \right) \left( \alpha -\gamma
\right) }+\frac{\beta ^{n}}{\left( \beta -\alpha \right) \left( \beta
-\gamma \right) }+\frac{\gamma ^{n}}{\left( \gamma -\alpha \right) \left(
\gamma -\beta \right) }.  \label{1.2}
\end{equation}%
The following relations hold: 
\begin{equation}
\alpha +\beta +\gamma =P,\;\alpha \beta \gamma =R,\ \alpha \beta +\beta
\gamma +\alpha \gamma =-Q,  \label{1.1}
\end{equation}%
where $\alpha $, $\beta $ and $\gamma $ are roots of equation (\ref{1}%
).\qquad

Recently, the norms of some special matrices were studied. For instance,
Solak \cite{7}, defined \ the $n\times n$ circulant matrices with Fibonacci
and Lucas numbers entries. Additionally, he investigated the upper and lower
bounds of spectral norms of the circulant matrices. Shen and Cen, in \cite{6}%
, have found upper and lower bounds for the spectral norms of $r$-circulant
matrices and obtained some bounds for the spectral norms of Kronecker and
Hadamard products of these matrices. In \cite{4}, the norms, eigenvalues and
determinants of some matrices related with special numbers were studied. In 
\cite{3}, Ipek obtained the spectral norms of circulant matrices with
Fibonacci and Lucas numbers entries. In \cite{8}, the authors computed the
spectral norm, eigenvalues and determinant of circulant matrix involving the
Generalized $k$-Horadam numbers.

Here, we define the Generalized Tribonacci-Lucas numbers and present
eigenvalues, spectral norms and determinants of circulant matrices with
Generalized Tribonacci and Generalized Tribonacci-Lucas numbers.

\section{Circulant Matrix with Generalized Tribonacci numbers}

From \cite{1}, we first remind that the circulant matrix $C=\left[ c_{ij}%
\right] \in M_{n,n}\left( 
\mathbb{C}
\right) $ is defined by 
\begin{equation*}
c_{ij}=\left\{ 
\begin{array}{c}
c_{j-i},\ \ \ \ \ \ \ j\geq i \\ 
c_{n+j-i},\ \ \ j<i%
\end{array}%
.\right.
\end{equation*}%
From \cite{2}, we further remind that, for a matrix $X=\left[ x_{i,j}\right]
\in M_{m,n}(%
\mathbb{C}
),$ the spectral norm of $X$ is given by 
\begin{equation*}
\left\Vert X\right\Vert _{2}=\sqrt{\max\limits_{1\leq i\leq n}\lambda
_{i}\left( X^{\ast }X\right) }
\end{equation*}%
where $X^{\ast }$ is the conjugate transpose of matrix $X.$

In this section, we define eigenvalues, the spectral norm and determinant of
the circulant matrix with the Generalized Tribonacci numbers entries. In
order to do that, we can also define \textit{the circulant matrix} as
follows.

\begin{definition}
A $n\times n$ circulant matrix with Generalized Tribonacci numbers entries
is defined by 
\begin{equation}
C_{n}(G)=\left[ 
\begin{array}{ccccc}
G_{1} & G_{2} & G_{3} & \cdots & G_{n} \\ 
G_{n} & G_{1} & G_{2} & \cdots & G_{n-1} \\ 
G_{n-1} & G_{n} & G_{1} & \cdots & G_{n-2} \\ 
\vdots & \vdots & \vdots & \ddots & \vdots \\ 
G_{2} & G_{3} & G_{4} & \cdots & G_{1}%
\end{array}%
\right] .  \label{1.3}
\end{equation}
\end{definition}

\begin{lemma}
\bigskip \lbrack 1] Let $A=circ(a_{0},a_{1},\cdots a_{n-1})$ be a $n\times n$
circulant matrix. Then we have%
\begin{equation*}
\lambda _{j}\left( A\right) =\sum_{k=0}^{n-1}a_{k}w^{-jk},
\end{equation*}%
where $w=e^{\frac{2\pi i}{n}},\ i=\sqrt{-1},\ j=0,1,\ldots n-1.$
\end{lemma}

\begin{lemma}
\lbrack 2] Let $A$ be an $n\times n$ matrix with eigenvalues $\lambda
_{1},\lambda _{2},\ldots ,\lambda _{n}.$ Then, $A$ is a normal matrix if and
only if the eigenvalues of $AA^{\ast }\ $are $\left\vert \lambda
_{1}\right\vert ^{2},\left\vert \lambda _{2}\right\vert ^{2},\ldots
,\left\vert \lambda _{n}\right\vert ^{2},$ where $A^{\ast }$ is the
conjugate transpose of the matrix $A$.
\end{lemma}

The following result gives us the eigenvalues of the matrix in (4).

\begin{theorem}
The eigenvalues of $C_{n}(G)\ $are%
\begin{equation*}
\lambda _{j}\left( C_{n}(G)\right) =\frac{G_{n+1}+\left(
G_{n+2}-PG_{n+1}-1\right) w^{-j}+RG_{n}w^{-2j}}{Pw^{-j}+Qw^{-2j}+Rw^{-3j}-1},
\end{equation*}%
where $G_{n}$ is the $n$th Generalized Tribonacci number and $P,\ Q,$ and $R$
are arbitrary integers.
\end{theorem}

\begin{proof}
By considering $X=\left( \alpha -\beta \right) \left( \alpha -\gamma \right)
,\ Y=\left( \beta -\alpha \right) \left( \beta -\gamma \right) ,\ Z=\left(
\gamma -\alpha \right) \left( \gamma -\beta \right) ,\ $from Lemma 2 and (%
\ref{1.2}), we have%
\begin{eqnarray*}
\lambda _{j}\left( C_{n}(G)\right)
&=&\sum_{k=0}^{n-1}G_{k+1}w^{-jk}=\sum_{k=0}^{n-1}\left( \frac{\alpha ^{k+1}%
}{X}+\frac{\beta ^{k+1}}{Y}+\frac{\gamma ^{k+1}}{Z}\right) w^{-jk} \\
&=&\frac{\alpha }{X}\left( \frac{\left( \alpha w^{-j}\right) ^{n}-1}{\alpha
w^{-j}-1}\right) +\frac{\beta }{Y}\left( \frac{\left( \beta w^{-j}\right)
^{n}-1}{\beta w^{-j}-1}\right) +\frac{\gamma }{Z}\left( \frac{\left( \gamma
w^{-j}\right) ^{n}-1}{\gamma w^{-j}-1}\right) \\
&=&\frac{\alpha ^{n+1}-\alpha }{X\left( \alpha w^{-j}-1\right) }+\frac{\beta
^{n+1}-\beta }{Y\left( \beta w^{-j}-1\right) }+\frac{\gamma ^{n+1}-\gamma }{%
Z\left( \gamma w^{-j}-1\right) }.
\end{eqnarray*}%
By using (\ref{1.2}) and (\ref{1.1}), we obtain%
\begin{equation*}
\lambda _{j}\left( C_{n}(G)\right) =\frac{G_{n+1}+\left(
G_{n+2}-PG_{n+1}-1\right) w^{-j}+RG_{n}w^{-2j}}{Pw^{-j}+Qw^{-2j}+Rw^{-3j}-1},
\end{equation*}%
as required.
\end{proof}

The following theorem gives us the spectral norm of $C_{n}(G)$.

\begin{theorem}
For matrix in (\ref{1.3}), we have%
\begin{equation*}
\left\Vert C_{n}(G)\right\Vert _{2}=\frac{G_{n+2}+\left( 1-P\right)
G_{n+1}+RG_{n}-1}{P+Q+R-1},
\end{equation*}%
where $P,\ Q,$ and $R$ are arbitrary integers and $P+Q+R-1\neq 0.$
\end{theorem}

\begin{proof}
We know that matrix $C_{n}(G)$ in (\ref{1.3}) is a normal matrix since $%
C_{n}(G)^{\ast }C_{n}(G)=C_{n}(G)C_{n}(G)^{\ast }$. Now, from Lemma 3, we
can write%
\begin{equation*}
\left\Vert C_{n}(G)\right\Vert _{2}=\sqrt{\left( \max\limits_{0\leq j\leq
n-1}\left\vert \lambda _{j}\left( C_{n}(G)\right) \right\vert ^{2}\right) }.
\end{equation*}%
In this last equality, if we take $j=0$, then $\lambda _{0}$ \ becomes \ the
maximum eigenvalue. Thus, $\left\Vert C_{n}(G)\right\Vert _{2}=\left\vert
\lambda _{0}\left( C_{n}(G)\right) \right\vert .$ In addition, Theorem 4, we
clearly obtain%
\begin{equation*}
\left\Vert C_{n}(G)\right\Vert _{2}=\frac{G_{n+2}+\left( 1-P\right)
G_{n+1}+RG_{n}-1}{P+Q+R-1}.
\end{equation*}%
Hence, proof is completed.
\end{proof}

In particular cases:

\begin{itemize}
\item Taking $P=Q=R=1$, then we get $\left\Vert C_{n}(T)\right\Vert _{2}=%
\frac{T_{n+1}+T_{n-1}-1}{2},\ $where$\ T_{n}$ is $n$th the classical
Tribonacci number.

\item Taking $P=0,\ Q=R=1$, then we get $\left\Vert C_{n}(A)\right\Vert
_{2}=A_{n+1}-1,\ $where$\ A_{n}$ is $n$th the Padovan number.

\item Taking $P=Q=1,~R=0$, then we get $\left\Vert C_{n}(F)\right\Vert
_{2}=F_{n+1}-1$ in \cite{3}$,\ $where$~F_{n}$ is $n$th the classical
Fibonacci number.

\item Taking $P=k,~Q=1,~R=0$, then we get $\left\Vert
C_{n}(F_{k})\right\Vert _{2}=\frac{F_{k.n}+F_{k,n-1}-1}{k},$\linebreak where$%
~F_{k,n}$ is $n$th the $k-$Fibonacci number, that is $%
F_{k,n}=kF_{k,n-1}+F_{k,n-2}$ with initial conditions $F_{k,0}=0$ and $%
F_{k,1}=1$.

\item Taking $P=2,\ Q=1,\ R=0$, then we get $\left\Vert C_{n}(B)\right\Vert
_{2}=\frac{B_{n+1}-B_{n}-1}{2},$\linebreak where$\ B_{n}$ is $n$th the
classical Pell number.

\item Taking $P=1,\ Q=2,\ R=0$, then we get $\left\Vert C_{n}(J)\right\Vert
_{2}=\frac{J_{n+1}-1}{2},\ $ where$\ J_{n}$ is $n$th the Jacobsthal number.
\end{itemize}

\begin{theorem}
The determinant \ of \ the matrix $C_{n}(G)=circ\left( G_{1},G_{2},\ldots
,G_{n}\right) $ is written by 
\begin{equation*}
\det (C_{n}(G))=\frac{\left( G_{n+1}\right) ^{n}\left(
1-K^{n}-L^{n}+K^{n}L^{n}\right) }{\left( -1\right) ^{n}\left(
1-S_{n+1}-R^{n}+M\right) },
\end{equation*}%
where%
\begin{eqnarray*}
K &=&\dfrac{PG_{n+1}-G_{n+2}+1-\sqrt{\left( PG_{n+1}-G_{n+2}+1\right)
^{2}-4RG_{n}G_{n+1}}}{2G_{n+1}}, \\
L &=&\dfrac{PG_{n+1}-G_{n+2}+1+\sqrt{\left( PG_{n+1}-G_{n+2}+1\right)
^{2}-4RG_{n}G_{n+1}}}{2G_{n+1}} \\
M &=&\alpha ^{n}\beta ^{n}+\beta ^{n}\gamma ^{n}+\alpha ^{n}\gamma ^{n}.
\end{eqnarray*}
\end{theorem}

\begin{proof}
From Theorem 4, we have 
\begin{eqnarray*}
\det (C_{n}(G)) &=&\prod\limits_{j=0}^{n-1}\lambda _{j}\left( C_{n}(G)\right)
\\
&=&\prod\limits_{j=0}^{n-1}\frac{G_{n+1}-\left( PG_{n+1}-G_{n+2}+1\right)
w^{-j}+RG_{n}w^{-2j}}{\left( \alpha w^{-j}-1\right) \left( \beta
w^{-j}-1\right) \left( \gamma w^{-j}-1\right) }.
\end{eqnarray*}%
By considering identities $\prod\limits_{k=0}^{n-1}\left( x-yw^{-k}\right)
=x^{n}-y^{n}$ and 
\begin{equation*}
\prod\limits_{k=0}^{n-1}\left( x-yw^{-k}+zw^{-2k}\right) =x^{n}\left( 
\begin{array}{c}
1-\left( \frac{y-\sqrt{y^{2}-4xz}}{2x}\right) ^{n}-\left( \frac{y+\sqrt{%
y^{2}-4xz}}{2x}\right) ^{n} \\ 
\ \ \ \ \ \ +\left( \frac{y-\sqrt{y^{2}-4xz}}{2x}\right) ^{n}\left( \frac{y+%
\sqrt{y^{2}-4xz}}{2x}\right) ^{n}%
\end{array}%
\right) ,
\end{equation*}%
\ we can write%
\begin{equation*}
\det (C_{n}(G))=\frac{\left( G_{n+1}\right) ^{n}\left(
1-K^{n}-L^{n}+K^{n}L^{n}\right) }{\left( -1\right) ^{n}\left(
1-S_{n+1}-R^{n}+M\right) }.
\end{equation*}
\end{proof}

\section{Circulant Matrix with Generalized Tribonacci-Lucas numbers}

In this section, we\ firstly define Generalized Tribonacci-Lucas sequence $%
\left\{ S_{n}\right\} _{n\in 
\mathbb{N}
}.$Then we compute the eigenvalues, norm and determinant of circulant matrix
with the Generalized Tribonacci-Lucas sequence $\left\{ S_{n}\right\} _{n\in 
\mathbb{N}
}.$

\begin{definition}
For $R\neq 0,\ $the Generalized Tribonacci-Lucas sequence $\left\{
S_{n}\right\} _{n\in 
\mathbb{N}
}$\ is defined by 
\begin{equation}
S_{n+3}=PS_{n+2}+QS_{n+1}+RS_{n},\ \ \ \ \ S_{0}=-\frac{Q}{R},\ S_{1}=3,\
S_{2}=P.  \label{1.25}
\end{equation}
\end{definition}

Binet formula enables us to state the Generalized Tribonacci-Lucas numbers.
It can be clearly obtained from the three distinct real roots $\alpha $, $%
\beta $ and $\gamma $ of characteristic equation of (\ref{1.25}) as$\ $%
\begin{equation}
S_{n}=\alpha ^{n-1}+\beta ^{n-1}+\gamma ^{n-1}.  \label{1.26}
\end{equation}

Note that, $\alpha $, $\beta $ and $\gamma $ are the same in (\ref{1.1}).

\begin{definition}
The $n\times n$ circulant matrix with the Generalized Tribonacci numbers
entries is defined by
\end{definition}

\begin{equation}
C_{n}(S)=\left[ 
\begin{array}{ccccc}
S_{1} & S_{2} & S_{3} & \cdots & S_{n} \\ 
S_{n} & S_{1} & S_{2} & \cdots & S_{n-1} \\ 
S_{n-1} & S_{n} & S_{1} & \cdots & S_{n-2} \\ 
\vdots & \vdots & \vdots & \ddots & \vdots \\ 
S_{2} & S_{3} & S_{4} & \cdots & S_{1}%
\end{array}%
\right] .  \label{1.5}
\end{equation}
The following result gives us the eigenvalues of the matrix in (\ref{1.5}).

\begin{theorem}
For $w=e^{\frac{2\pi i}{n}},\ i=\sqrt{-1}$ and $j=0,1,\ldots ,n-1,$ the
eigenvalues of $C_{n}(S)\ $are%
\begin{equation*}
\lambda _{j}\left( C_{n}(S)\right) =\frac{S_{n+1}-3+\left(
S_{n+2}-PS_{n+1}+2P\right) w^{-j}+\left( RS_{n}+Q\right) w^{-2j}}{%
Pw^{-j}+Qw^{-2j}+Rw^{-3j}-1},
\end{equation*}%
where $S_{n}$ is the $n$th Generalized Tribonacci-Lucas number and $P,\ Q,$
and $R$ are arbitrary integers.
\end{theorem}

\begin{proof}
By considering Lemma 2 and (\ref{1.26}), we have%
\begin{eqnarray*}
\lambda _{j}\left( C_{n}(S)\right)
&=&\sum_{k=0}^{n-1}S_{k+1}w^{-jk}=\sum_{k=0}^{n-1}\left( \alpha ^{k}+\beta
^{k}+\gamma ^{k}\right) w^{-jk} \\
&=&\frac{\alpha ^{n}-1}{\alpha w^{-j}-1}+\frac{\beta ^{n}-1}{\beta w^{-j}-1}+%
\frac{\gamma ^{n}-1}{\gamma w^{-j}-1}.
\end{eqnarray*}%
By using (\ref{1.1})$,\ $ we obtain%
\begin{equation*}
\lambda _{j}\left( C_{n}(S)\right) =\frac{S_{n+1}-3+\left(
S_{n+2}-PS_{n+1}+2P\right) w^{-j}+\left( RS_{n}+Q\right) w^{-2j}}{%
Pw^{-j}+Qw^{-2j}+Rw^{-3j}-1}.
\end{equation*}
\end{proof}

\begin{theorem}
For the matrix in (\ref{1.5}), the following equalities hold:

\begin{description}
\item[\textbf{i)}] $\left\Vert C_{n}(S)\right\Vert _{2}=\frac{S_{n+2}+\left(
1-P\right) S_{n+1}+RS_{n}+2P+Q-3}{P+Q+R-1},$

\item where $P,\ Q,$ and $R$ are arbitrary integers and $P+Q+R-1\neq 0.$

\item[\textbf{ii)}] $\det (C_{n}(S))=\frac{\left( S_{n+1}-3\right) ^{n}\left[
1-C^{n}-D^{n}+C^{n}D^{n}\right] }{\left( -1\right) ^{n}\left[
1-S_{n+1}-R^{n}+E\right] },$

\item where 
\begin{eqnarray*}
C &=&\dfrac{S_{n+2}-PS_{n+1}+2P-\sqrt{\left( S_{n+2}-PS_{n+1}+2P\right)
^{2}-4\left( S_{n+1}-3\right) \left( RS_{n}+Q\right) }}{2\left(
S_{n+1}-3\right) }, \\
D &=&\dfrac{S_{n+2}-PS_{n+1}+2P+\sqrt{\left( S_{n+2}-PS_{n+1}+2P\right)
^{2}-4\left( S_{n+1}-3\right) \left( RS_{n}+Q\right) }}{2\left(
S_{n+1}-3\right) }, \\
E &=&\alpha ^{n}\beta ^{n}+\beta ^{n}\gamma ^{n}+\alpha ^{n}\gamma ^{n}.
\end{eqnarray*}
\end{description}
\end{theorem}

\begin{proof}
Proof of this theorem can be seen by using Theorem 9.
\end{proof}

In particular cases:

\begin{itemize}
\item Taking $P=Q=R=1$, then we get $\left\Vert C_{n}(Y)\right\Vert _{2}=%
\frac{Y_{n+1}+Y_{n-1}}{2},\ $where$\ Y_{n}$ is $n$th the classical
Tribonacci-Lucas number.

\item Taking $P=0,\ Q=R=1$, then we get $\left\Vert C_{n}(Z)\right\Vert
_{2}=Z_{n+4}-2,\ $where$\ Z_{n}$ is $n$th the Perrin number.
\end{itemize}

\end{document}